\documentclass[11pt]{amsart}
\usepackage{amsmath,amssymb, graphicx, amscd,latexsym,color,comment}

\makeatletter
\newtheorem{Theorem}{Theorem}

\newtheorem{Corollary}[Theorem]{Corollary}
\newtheorem{Proposition}[Theorem]{Proposition}

\newtheorem{Remark}[Theorem]{Remark}

\newtheorem{Example}[Theorem]{Example}

\newtheorem{Assertion}[Theorem]{Assertion}
\newtheorem{Conjecture}[Theorem]{Simply connectedness Conjecture}

\newcommand{\eps}{\varepsilon}

\newcommand\la{\lambda}
\newcommand\vphi{\varphi}

\newcommand\si{\sigma}

\newcommand\ga{\gamma}

\newcommand\de{\delta}

\newcommand\BC{ {\mathbb C}}

\newcommand\BZ{{\mathbb  Z}}
\newcommand\BR{ {\mathbb  R}}
\newcommand\BP{ {\mathbb  P}}

\newcommand\bfu{\mbox {\bf  u}}

\newcommand\bfz{\mbox {\bf  z}}

\newcommand\nl{\newline}

\newcommand\discrim{\rm{ discrim}\/}

\newcommand\id{\rm{id}}

\newcommand\inv{^{-1}}

\def\mapright#1{\smash{\mathop{\longrightarrow}\limits^{{#1}}}}

\def\inv{^{-1}}

\begin{document}

\title[ Smooth mixed projective curves and a  conjecture 
]
{
Smooth mixed projective curves and a  conjecture }

\author
[M. Oka ]
{Mutsuo Oka \\\\
{\tiny Dedicated to Professor Egbert Brieskorn}}
\address{\vtop{
\hbox{Department of Mathematics}
\hbox{Tokyo  University of Science}
\hbox{Kagurazaka 1-3, Shinjuku-ku}
\hbox{Tokyo 162-8601}
\hbox{\rm{E-mail}: {\rm oka@rs.kagu.tus.ac.jp}}
}}

\begin{abstract}
Let $f(\bfz,\bar\bfz)$ be a   strongly  mixed homogeneous polynomial
of $3$ variables $\bfz=(z_1,z_2,z_3)$ of polar degree $q$ with an isolated singularity at the origin. It
 defines a smooth Riemann surface 
$C$
in the complex projective space
$\BP^{2}$.
The fundamental group of the complement $\pi_1(\mathbb P^2\setminus C)$ is a cyclic  group of order $q$ if
$f$ is a homogeneous polynomial without $\bar \bfz$. We propose  a conjecture that this may be even true for mixed homogeneous polynomials
by giving several supporting examples.
\end{abstract}
\maketitle

\maketitle

\section{Introduction}
Let $f(\bfz,\bar\bfz)=\sum_{\nu,\mu}c_{\nu,\mu}\bfz^{\nu}{\bar\bfz}^\mu$ be a mixed polynomial of $n$-variables  $\bfz=(z_1,\dots, z_n)\in \BC^n$.  $f(\bfz,\bar \bfz)$ is called a
{\em  strongly mixed   homogeneous polynomial }
of $n$-variables  $\bfz=(z_1,\dots, z_n)\in \BC^n$ 
with polar degree $q$ and radial
degree
$d$ if $|\nu|+|\mu|=d$ and $|\nu|-|\mu|=q$ for  any $\nu,\mu$ with  $c_{\nu,\mu}\ne 0$.
For such a polynomial, we consider the canonical $\mathbb C^*$-action on $\mathbb C^n$.
Recall that a strongly mixed  homogeneous polynomial
$f(\bfz,\bar\bfz)$
satisfies the equality (\cite{MC}):
\begin{eqnarray*}\label{action}
 f(\lambda\circ\bfz,\overline{\lambda\circ\bfz})=r^{d}\exp{(iq\theta)}f(\bfz,\bar\bfz),\quad\text{where}\,\,
 \lambda= r\exp(i\theta)\in \BR^+\times S^1.
\end{eqnarray*}
Here 
$\lambda\circ\bfz\,=\,(\la z_1,\dots,\la z_n)$. By the above equality,
it defines canonically a real analytic subvariety $V$ of real codimension 2 in 
$\mathbb P^{n-1}$:
\[
V=\{[\bfz]\in \mathbb P^{n-1}\,|\, f(\bfz,\bar\bfz)=0\}.
\]
Let $\widetilde V$ be the mixed affine hypersurface
\[
 \widetilde V=f\inv(0)=\{\bfz\in \BC^n\,|\, f(\bfz,\bar\bfz)=0\}.
\]
We assume that $\widetilde V$ has an isolated singularity at the origin, or equivalently $V$ is a non-singular variety. 
Let $f:\BC^n\setminus \widetilde V\to \BC^*$  be the global 
Milnor fibration defined by $f$
and let $F$ be the Milnor fiber.
Namely $F=f\inv(1)\subset \BC^n$.
The monodromy map $h:F\to F$  is defined by
\[
 h(\bfz)=(\omega_q z_1,\dots,\omega_q z_n),
\quad \omega_q=\exp(\frac{2\pi\,i}{q}). 
\]
and its restriction  of the Hopf fibration to the Milnor fiber
$\pi:\,F\to \BP^{n-1}\setminus V$ is nothing but the quotient map by the cyclic action induced by $h$.

In \cite{OkaJarc,MC}, we have shown  that 
\begin{Theorem}[Theorem 11, \cite{MC}]
The embedding degree of $V$ is equal to the polar degree $q$.
In particular, $H_1(\mathbb P^{n-1}\setminus V)=\mathbb Z/q\mathbb Z$ if $V$ is non-singular in an open dense subset.
\end{Theorem}
\begin{Proposition}
The Euler characteristics satisfy the following equalities.
 \begin{enumerate}\label{formula1}
\item
$\chi(F)=q\,\chi(\BP^{n-1}\setminus V)$ and 
    $\chi(\BP^{n-1}\setminus V)=n-\chi(V)$.
     In particular, if $n=3$ and $V$ is smooth curve with the genus  $g$,
then $\chi(F)=q(1+2g)$.

\item The following sequence is exact.
$$1\to \pi_1(F)\mapright{\pi_{\sharp}} \pi_1(\BP^{n-1}\setminus V)\to
     \BZ/q\BZ\to 1.$$
      \item
If  $q=1$, the projection $\pi:F\to\BP^{n-1}\setminus V$ is a
 diffeomorphism.
\end{enumerate}\end{Proposition}

Using the periodic monodromy argument in \cite{Milnor}, we have 
\begin{Proposition}
The zeta function of the monodromy $h:F\to F$ is given by
\[
 \zeta(t)=(1-t^{q})^{-\chi(F)/q}.
\]In particular, if $q=1$,
$h=\id_F$ and 
$\zeta(t)=(1-t)^{-\chi(F)}$.
\end{Proposition}
If $f$ is a holomorphic function, $F$ is $(n-2)$-connected and it is homotopic to
 a bouquet of $\mu$ spheres of dimension $n-1$
(\cite{Milnor}). 

For mixed polynomials, we do not have any connectivity theorem.  
But we do not have any examples  which breaks the connectivity as the holomorphic case.
Thus we propose the following conjecture as a first working problem.
\begin{Conjecture}\label{simply-connected-conjecture}
Assume $n=3$ and that $f$ is a   non-degenerate, strongly mixed homogeneous polynomial of polar degree $q$. Then 
Milnor fiber $F$ is simply connected. Equivalently
the fundamental group of the complement $\pi_1(\mathbb P^2 -  V)$ is a cyclic group of order $q$.
\end{Conjecture}
The purpose of this paper is to give several non-trivial examples which support this conjecture.

\section{Easy mixed polynomials } 
Unlike the holomorphic case, we do not know in general the connectivity of the Milnor fiber even under the  assumption that  $V$ has an isolated singularity at the origin. 
In this section, we study easy examples.
Suppose that $f$ is either a simplicial mixed polynomial or a join type or twisted join type  polynomial. Then  the connectivity 
behaves just as the holomorphic case.
 We will first explain these polynomials below.
 \subsection{Simplicial polynomial} Assume that $n=3$ and $\bfz=(z_1,z_2,z_3)$.  A mixed polynomial $f(\bfz,\bar\bfz)$ is called simplicial if  it is a linear  sum of three mixed monomials
 \[
 f(\bfz,\bar\bfz)=\sum_{i=1}^3 c_i \bfz^{\nu_i}{\bar\bfz}^{\mu_i}
 \]
 and 
 two matrices
 \[
 \left(\nu_i\pm \mu_i\right)_{i=1}^3=\left( 
 \begin{matrix}
 \nu_{11}\pm\mu_{11}&\nu_{12}\pm\mu_{12}&\nu_{13}\pm\mu_{13}\\
 \nu_{21}\pm\mu_{21}&\nu_{22}\pm\mu_{22}&\nu_{23}\pm\mu_{23}\\
 \nu_{31}\pm\mu_{31}&\nu_{32}\pm\mu_{32}&\nu_{33}\pm\mu_{33}\\
 \end{matrix}\right)
 \]
are non-degenerate
 where $\nu_i=(\nu_{i1},\nu_{i2},\nu_{i3}),\,\mu_i=(\mu_{i1},\mu_{i2},\mu_{i3})$. In this case, we may assume that $c_i=1$ for $i=1,2,3$. Among them, the following polynomials are strongly mixed homogeneous and  have an isolated singularity at the origin. 
  \begin{eqnarray*}
& f_B&:=z_1^{q+r}{\bar z_1}^{r}+z_2^{q+r}{\bar z_2}^{r}+z_3^{q+r}{\bar z_3}^{r},\,\,\text{(Brieskorn Type)}\\
  &f_I&:=z_1^{q+r-1}{\bar z_1}^{r}z_2+z_2^{q+r-1}{\bar z_2}^{r}z_3+z_3^{q+r}{\bar z_3}^{r},\,\, \text{(Tree type a)}\\
  &f_{II}&:=z_1^{q+r-1}{\bar z_1}^{r}z_2+z_2^{q+r-1}{\bar z_2}^{r}z_3+z_3^{q+r-1}{\bar z_3}^{r}z_1,\,\, \text{(Cyclic type a)}\\
   &f_{III}&:=z_1^{q+r-1}{\bar z_1}^{r}z_2+z_2^{q+r-1}{\bar z_2}^{r}z_1+z_3^{q+r}{\bar z_3}^{r},\,\,\text{(Simplicial+Join a)}\\
   & f_I'&:=z_1^{q+r}{\bar z_1}^{r-1}\bar z_2+z_2^{q+r}{\bar z_2}^{r-1}{\bar z_3}+z_3^{q+r}{\bar z_3}^{r},\,\,
               \text{(Tree type b)}\\
 &f_{II}'&:=z_1^{q+r}{\bar z_1}^{r-1}\bar z_2+z_2^{q+r}{\bar z_2}^{r-1}\bar z_3+z_3^{q+r}{\bar z_3}^{r-1}\bar z_1,\,\,
 \text{(Cyclic type b)}\\
 &f_{III}'&:=z_1^{q+r}{\bar z_1}^{r-1}\bar z_2+z_2^{q+r}{\bar z_2}^{r-1}\bar z_1+z_3^{q+r}{\bar z_3}^{r},\,\,\text{(Simplicial+Join b)}.
 \end{eqnarray*}
 Here $q\ge 2$ and $r\ge 1$ are positive integers.
 All above polynomials have simply connected Milnor fibers (\cite{OkaPolar}).
 For $f_B$, $f_I$, $f_{II}$, $f_{III}$, their Milnor fiberings and links are in fact isotopic to the holomorphic ones 
 by the contraction $z_i^{r}{\bar z_i}^r\mapsto 1$ (\cite{OkaBrieskorn,Inaba}):
  \begin{eqnarray*}
& f_B&:=z_1^{q}+z_2^{q}+z_3^{q},\,\,\text{(Brieskorn Type)}\\
  &f_I&:=z_1^{q-1}z_2+z_2^{q-1}z_3+z_3^{q},\,\, \text{(Tree type a)}\\
 &f_{II}&:=z_1^{q-1}z_2+z_2^{q-1}z_3+z_3^{q-1}z_1,\,\, \text{(Cyclic type a)}\\
 &f_{III}&:=z_1^{q-1}z_2+z_2^{q-1}z_1+z_3^{q},\,\,\text{(Simplicial+Join a)}.
  \end{eqnarray*}
  \begin{Remark}The above list does not cover all. For example, we can combine $f_I$ and $f_I'$:
  \[{f_I}'':=z_1^{q+r}{\bar z_1}^{r-1}\bar z_2+z_2^{q+r-1}{\bar z_2}^{r}z_3+z_3^{q+r}{\bar z_3}^{r}.
  \]
\end{Remark}
\subsection{Join type mixed polynomials}
Let $f(\bfz,\bar \bfz)$ be a strongly mixed homogeneous  convenient polynomial of $n$-variables $\mathbf z=(z_1,\dots,z_n)$ of polar degree $q$ and radial degree $q+2r$.
 with an isolated singularity at the origin. 
Consider the join polynomial $g:=f(\bfz,\bar \bfz)-w^{q+r}{\bar w}^r$ of $(n+1)$-variables. Put $F_f, F_g$ be the  respective 
Milnor fibers of $f$ and $g$. Consider
 the projective
Mixed hypersurfaces $V_f$ and $V_g$ defined by $f=0$ and  $g=0$ respectively in $\mathbb P^{n-1}$ or $\mathbb P^n$.
\begin{Theorem} 
Assume that 
there is a smooth point in $V_f$
and $n\ge 2$.  Then
 \begin{enumerate} 
 \item  $F_f$ is connected and
 \item 
  $\pi_1(\mathbb P^n\setminus V_g)=\mathbb Z/q\mathbb Z$.
 \end{enumerate}
\end{Theorem}
\begin{proof} In this theorem, we do not assume that $f$ is strongly non-degenerate.
Note that 
\[\begin{split}
F_f=&\{\bfz\in \mathbb C^n\,|\, f(\bfz,\bar \bfz)-1=0\}\\
V_f&=\{[\mathbf z]\in \mathbb P^{n-1}\,|\, f(\mathbf z,\bar{\mathbf z})=0\}\\
V_g=&\{[\bfz:w]\in \mathbb P^n\,|\, f(\bfz,\bar\bfz)-w^{q+r}{\bar w}^r=0\}.
\end{split}
\]
Consider the affine coordinate $U_w:=\{w\ne 0\}$ in $\mathbb P^n$. In this coordinate space, using affine coordinates $u_j=z_j/w,\,j=1,\dots,n$, we see that
\[
V_g\cap U_w=\{\bfu\in \mathbb C^n\,|\, f(\bfu,\bar\bfu)-1=0\}.
\]
This expression says that $F_f\cong V_g\cap U_w$.  Note that $V_g\cap\{w=0\}\cong V_f$ has a smooth point $p$.
Consider the projection
$\pi: \mathbb P^n\to \mathbb P^{n-1}$ which is defined by $[\bfz:w]\mapsto [\bfz]\in \mathbb P^{n-1}$. 
Then the restriction
$\pi:  V_g\to \mathbb P^{n-1}$ is $q$-fold covering branched over $V_f\subset \mathbb P^{n-1}$.
Take a non-singular point  $p$  of $V_f\subset \mathbb P^{n-1}$
and consider a small normal disk $D$ centered at $p$. For simplicity, we assume that
$p\in\{z_1\ne 0\}$ and we choose affine coordinate chart
$\{z_1\ne 0\}$ with affine coordinates
$v_j=z_j/z_1,\,j=2,\dots, n$ and $x=w/z_1$. In this chart, $V_g$ is defined by
$f(\mathbf v,\bar{\mathbf v})-x^{q+r}{\bar x}^r=0$ with $\mathbf v=(1,v_2,\dots, v_n)$.
Then the covering  $(\mathbf v,x)\mapsto  \mathbf v$ is topologically equivalent to the holomorphic cyclic covering defined by 
$x^q-f=0$ in a small disk $D$ with center $p$. (In $D$ we can take the function $f:D\to\mathbb C$ as a real analytic complex-valued coordinate function and we may assume that the image $f(D)$ is a small unit disk $\Delta_\rho$ with radius $\rho$.)
Thus the fiber of a boundary point $p'$ , $f(p')=\rho e^{i\theta_0}\in \partial \Delta$   is distributed as
$\{Re^{i(\theta_0+2j\pi)/q}\,|\, j=0,\dots, q-1\}$ in $x$-coordinate with $R=|\rho|^{1/(q+2r)}$ and under the local monodromy
along $\partial D$, they are rotated  by 
$2\pi/q$ angles. 
Thus 
$\pi^{-1}(D^*)$ is connected, where  $D^*:=D\setminus\{p\}$.  As  $\mathbb P^{n-1}\setminus V_f$ is connected,
 any  point $y\in V_g\setminus V_f$ can be connected using the covering structure to one of the  points
 $\pi^{-1}(p')$.
 Here we identify $V_f$ with $V_g\cap \{w=0\}$. As $V_g-V_f=V_g\cap U_\omega$,  $V_g\cap U_\omega$ is connected.

Now we consider the fundamental group, assuming $n=2$ for  simplicity.
$V_g$ is defined by $z_3^{q+r}{\bar z_3}^r-f(\mathbf z,\bar{\mathbf z})$ where $\mathbf z=(z_1,z_2)$.
Consider the pencil lines $L_\eta=\{z_2=\eta z_1\}$ and let $b=(0:0:1)$ be the base point of the pencil.
 Let $\widetilde {\mathbb P}^2$ be the blow-up space at $b$.
 Then  $\widetilde \pi:\widetilde {\mathbb P}^2\to \mathbb P^{1}$ is well defined and $\pi_1(\widetilde {\mathbb P}^2\setminus \widetilde V_g)\equiv \pi_1(\mathbb P^2\setminus V_g)$  with $\widetilde V_g={\widetilde \pi}\inv(V_g)\cong V_g$.
The zero points $f(\bfz,\bar\bfz)=0$ are the locus of singular pencil lines. Take a simple zero $p\in V_f$ and take $p'$ nearby as a base line
and put $L=\pi^{-1}(p')$.
Take generators $\xi_1,\dots,\xi_q$ of $\pi_1(L\setminus V_g\cap L)$ as in Figure \ref{Generators}. The centers
of the small circles are the points of
$L\cap V_g$. We always orient the small circles counterclockwise.
Then the monodromy relations at $p$ is given by
\[
\xi_1=\xi_2=\dots=\xi_q,\,\quad\xi_q\dots\xi_1=e
\]
See \cite{OkaSurvey}. The argument is the exactly same with that of complex algebraic curve with a maximal flex point in Zariski \cite{Za1}.
Thus we get $\xi_1^q=e$ and 
 $\pi_1(\mathbb P^2\setminus V_g)\cong \mathbb Z/q\mathbb Z$.

The assertion 2 is true for any $n\ge 2$. For $n>2$, we take a generic hyperplane $H$ of type $a_1z_1+\dots+a_nz_n=0$
which contains $[0:\dots:0:1]$ and use the surjectivity $\pi_1(H\setminus V\cap H)\to\pi_1(\mathbb P^n\setminus V_g)$.
The defining polynomial of $V_g\cap H$ is also of join type and use an induction argument. Here we do not use the Zariski Hyperplane section theorem \cite{Hamm-Le} (we do not know if the same assertion   holds for mixed hypersurface or not) but we only use the surjectivity  for a non-singular mixed hypersurface of join type which is easy to be shown. We leave this assertion to the readers.
\end{proof}
\begin{figure}[htb]
\setlength{\unitlength}{1bp}
\begin{picture}(600,250)(-250,-620) 
{\includegraphics{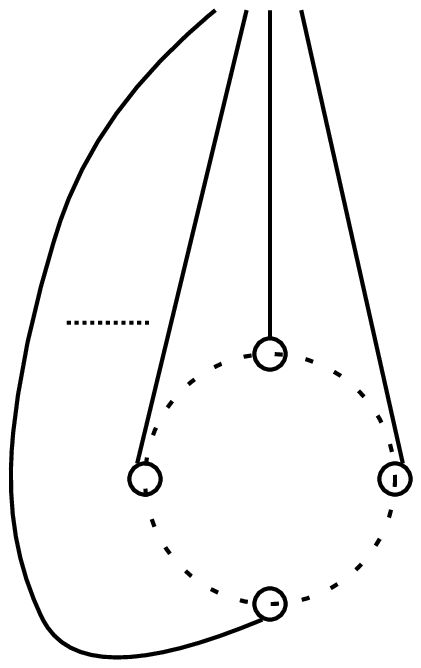}}
\put(-110,-490){$\xi_1$}
\put(-110,-523){$\cdot$}
\put(-140,-470){$\xi_2$}
\put(-145,-486){$\cdot$}
\put(-170,-485){$\xi_{3}$}
\put(-181,-522){$\cdot$}
\put(-230,-540){$\xi_q$}
\put(-146,-559){$\cdot$}
\end{picture}
\caption{Generators of $\pi_1(L-L\cap V)$}\label{Generators}
\end{figure}
\begin{Example}\label{Rhie}
\end{Example}
Consider the Rhie's Lens equation
\[
\vphi_n(z):=\bar z-\frac{z^{n-2}}{z^{n-1}-a^{n-1}}-\frac{\eps}{z}=\frac{g(z,\bar z)}{(z^{n-1}-a^{n-1}) z}=0,\,n\ge 2.
\]
We can choose suitable positive numbers $a,\eps$ so that $0<\eps\ll a\ll 1$ and 
$\vphi_n$ has $5(n-1)$ simple zeros  (see \cite{OkaLens}). 
Let $g(z,\bar z)$ be the numerator of $\vphi_n$ 
and take the homogenization of $g(z,\bar z)$
\begin{multline*}
G(\mathbf z',\bar {\mathbf z'}):=g(z_1/z_2,\bar z_1/\bar z_2)z_2^n \bar z_2\\
= \bar z_1(z_1^n-a^{n-1} z_1z_2^{n-1})-\bar z_2\left (z_1^{n-1}z_2+\eps(z_1^{n-1}z_2-a^{n-1}z_2^{n-1}z_1)  \right)
\end{multline*}
where $\mathbf z'=(z_1,z_2)$.
Consider the join type polynomial and the associated projective curve $C$:
\[
C:\,f(\bfz,\bar{\bfz}):=z_3^n\bar z_3+G(\mathbf z',\bar{\mathbf z'})=0,\quad \mathbf z=(z_1,z_2,z_3).
\]
Observe that $f$ is strongly mixed homogeneous of polar degree $q=n-1$ and radial degree $n+1$.
Consider the affine chart $\{z_2\ne 0\}$ and consider the affine coordinates
$w_3=z_3/z_2,w_1=z_1/z_2$. Then the affine equation takes the form
$w_3^{n}\bar w_3-g(w_1,\bar w_1)=0$. Consider the pencil of lines $L_\eta=\{z_1-\eta z_2=0\}$ or in the affine equation,
$w_1=\eta$.
There are exactly $5(n-1)$ singular pencil lines corresponding to the zeros of $g(w_1,\bar w_1)=0$.
These roots are all simple by the construction.
In a small neighborhood of  any such zero, the projection $\pi:C\to \mathbb C$ is locally equivalent to $w_3^{n}\bar w_3-w_1=0$ or $w_3^{n}\bar w_3-\bar w_1=0$
depending the sign of the zero. Taking a  point $\eta_0$ near some zero of
$g$ and on the line $L_{\eta_0}$, take generators $\xi_1,\dots,\xi_q$ of
 $\pi_1(L_{\eta_0}\setminus C)$ as in Figure \ref{Generators}, we get
that $\xi_1=\dots=\xi_q$ as the monodromy relation. Thus we get $\pi_1(\mathbb P^2\setminus C)=\mathbb Z/q\mathbb Z$.
Note that $L_\infty\cap C$ consists of $q$ simple points. Thus
the Euler number and the genus of $C$ are calculated easily as
\begin{eqnarray*}
\chi(C)&=&(n-1)(2-(5n-5)-1)+5n-5+n-1=17n-5n^2-12,\,\\
g(C)&=&\frac{(5n-7)(n-2)}2.\qquad\qquad
\end{eqnarray*}
In the moduli space of mixed polynomial of polar degree $n-1$ and radial degree $n+1$, the lowest genus is taken by
$z_1^{n}\bar z_1+z_2^n\bar z_2+z_3^n\bar z_3$ which is isotopic to the holomorphic curve of degree $n-1$ and therefore the genus is $(n-2)(n-3)/2$ by Pl\"ucker's formula.
\subsection{Twisted join type polynomials}
Let $f(\bfz)$ be a strongly mixed homogeneous polynomial of polar degree $q$ and radial degree $q+2r$ and consider
 the mixed homogeneous polynomial of $(n+1)$-variables:
\[
 g(\bfz,\bar\bfz,w,\bar w)\,=\, f(\bfz,\bar\bfz)+
\bar z_n w^{q+r}{\bar w}^{r-1}.
\]
$g$ is also strongly mixed homogeneous polynomial.
Recall that 
$f(\bfz,\bar \bfz)$ is  called to be {\em 1-convenient} if
the restriction of $f$ to each coordinate subspace
$f_i:=f|_{\{z_i=0\}}$ is non-trivial for  $i=1,\dots,n$ (\cite{OkaPolar})
  \begin{Theorem}\label{Jarc} \text{(\cite{OkaJarc})}
  Assume that $n\ge 2$ and  $f$ is 1-convenient with a connected Milnor fiber $F_f$
  and let
 $g(\bfz,\bar\bfz,w,\bar w)$ be the twisted join polynomial as above.
  \begin{enumerate}
       \item
            The Milnor fiber  $ F_g=g\inv(1)$ of $g$ is simply connected.
            \item
            The Euler characteristic of $ F_g$ is given by the formula:
            \[
             \chi(F_g)=-(q+r)\chi(F_f)+(q+r+1)\chi(F_{f_n})
                       \]
                       where $f_n:=f|\{z_n=0\}$ and $F_{f_n}=f_n\inv(1)$.
\end{enumerate}
  \end{Theorem}
Assume that $n=2$ and $f(\bfz,\bar\bfz)$ has an isolated singularity at the origin. Then we have
\begin{Corollary}$V=\{g=0\}\subset \mathbb P^2$ is a non-singular 
mixed curve and 
\newline
$\pi_1(\mathbb P^2-V)\cong \mathbb Z/q\mathbb Z$.
\end{Corollary}
\begin{Example}
\end{Example}Consider mixed curve defined by
\[
f_I'=z_1^{q+r}{\bar z_1}^{r-1}\bar z_2+z_2^{q+r}{\bar z_2}^{r-1}{\bar z_3}+z_3^{q+r}{\bar z_3}^{r},\,\,
               \text{(Tree type b)}
               \]
As $f_I'$ is simplicial
and also of twisted join type as 
$f_I'=z_1^{q+r}{\bar z_1}^{r-1}\bar z_2+(z_2^{q+r}{\bar z_2}^{r-1}{\bar z_3}+z_3^{q+r}{\bar z_3}^{r})$, we show that Milnor fiber is simply connected and $\pi_1(\mathbb P^2\setminus C)\cong \mathbb Z/q\mathbb Z$.
 Here as $(z_2^{q+r}{\bar z_2}^{r-1}{\bar z_3}+z_3^{q+r}{\bar z_3}^{r})$ is not 1- convenient,
 Theorem \ref{Jarc} can not be applied directly.
Let us see this assertion directly. We take the coordinate chart $U_2:=\{z_2\ne 0\}$ and put $w_1=z_1/z_2, w_3=z_3/z_2$. 
Then affine equation of $C$ in $U_2$
is
\[
f(w_1,w_3)=w_1^{q+r}{\bar w_1}^{r-1}+\bar w_3+ w_3^{q+r}{\bar w_3}^r.
\]
We consider the pencil $L_\eta:=\{w_3-\eta=0\},\,\eta\in \mathbb C$.
It is easy to see the branching locus is  $q+2$ points given by
\[
\Sigma:=\{w_3\,|\, \bar w_3(w_3^{q+r}{\bar w_3}^{r-1}+1)=0\}. 
\]
The base point of the pencil is $b=[1:0:0]$ and note that $b\in C$.
$L_\eta\cap C$ is $q+1$ points over $\mathbb C\setminus \Sigma$ and 1 point over $\Sigma$.
Taking a generic pencil $L_{\eta_0}$ near a branching point $w\in \Sigma$ and take  generators
$\xi_1,\dots, \xi_{q+1}$ of $\pi_1(L_{\eta_0}\setminus C)$ similarly as those in Figure \ref{Generators},
we get cyclic monodromy relations at each point of $\Sigma$:
\[
\xi_1=\xi_2=\cdots=\xi_{q+1}.
\]
This is enough to conclude that $\pi_1(\mathbb P^2\setminus C)$ is abelian and therefor
isomorphic to
$H_1(\mathbb P^2\setminus C)\cong \mathbb Z/q\mathbb Z$.
As for the Euler characteristic, we get $\chi(C)=-(q+1)(q-1)+q+3=-q^2+q+4$. Thus the genus of $C$ is $(q+1)(q-2)/2$.

\section{Non-trivial examples}
Let $F(\bfz,\bar\bfz)$ be a strongly non-degenerate mixed homogeneous polynomial
of three variables $\bfz=(z_1,z_2,z_3)$ of polar degree $q$ and radial degree $q+2r$
and we consider the projective mixed curve
\[
C:=\{[\bfz]\in \mathbb P^2\,|\, F(\bfz,\bar \bfz)=0\}.
\]
We study the geometric  structure of $C$ and the fundamental group $\pi_1(\mathbb P^2-C)$ using  the pencil $L_\eta:=\{z_2=\eta z_3\},\,\eta\in \mathbb C$, or equivalently the projection
\[
p: (\mathbb P^2,C)\to \mathbb P^1,\quad [\bfz]\mapsto [z_2,z_3].
\]
Take the affine coordinate $U_3:=\{z_3\ne 0\}$ with coordinate functions $(z,w)$ with $z=z_1/z_3,\,w=z_2/z_3$. Then $C\cap U_3$ is defined by 
$ f(z,w,\bar z,\bar w)=F(\bfz,\bar \bfz)/z_3^{q+r}{\bar z_3}^r=0$. Let $\Sigma\subset\mathbb P^1 $ be the branching locus of $p$.
\subsubsection{Holomorphic case} If $F$ is homogeneous polynomial without complex conjugate variables,
$\Sigma$ is described  by the discriminant locus of $f$ as a polynomial in $z$. Put
$\si(w):=\discrim_z  f(z,w)$. Thus $\Sigma$ is a finite points $\Sigma=\{\rho_1,\dots,\rho_\ell\}$ given 
by $\sigma(w)=0$. For any  $\rho_j\in \Sigma$ and $\rho_{j,k}\in p^{-1}(\rho_j)$, $C$ is locally a cyclic covering of order $s_{j,k}$ 
at $  \rho_{j,k}$ where 
 $s_{j,k}$ is the multiplicity of $ \rho_{j,k}$ 
in $p^{-1}(\rho_j)$ as the root of $f(z,\rho_j)=0$ which is equal to the intersection multiplicity of $L_{\rho_j}$ and $C$ at $\rho_{j,k}$.

\subsubsection{Mixed polynomial case} Let $F$ be a mixed homogeneous polynomial.
Usually it is not easy to compute $\Sigma$.
Instead of computing $\Sigma$, we proceed as follows. Let $z=x+yi$ and $w=u+vi$ and write $f$ as 
$f(x,y,u,v):=g(x,y,u,v)+ih(x,y,u,v)$ where $g$ and $h$ are real polynomials which are the real and imaginary part of $f$ respectively.
Consider the complex algebraic variety 
\[
C(\mathbb C):= \{(x,y,u,v)\in\mathbb C^4 \,|g(x,y,u,v)=h(x,y,u,v)=0\}
\]
which is the complexification of our curve. Note that $C(\mathbb C)\cap \mathbb R^4=C$.
The branching locus of $p_{\mathbb C}: C(\mathbb C)\to \mathbb C^2$ is obtained by a Groebner basis calculation from the ideal
$[g,h,J]$
where $J=\frac{\partial g}{\partial x}\frac{\partial h}{\partial y}-
\frac{\partial g}{\partial y}\frac{\partial h}{\partial x}$ and $[g,h,J]$ is the ideal generated by $g,h,J$. The defining ideal is generated by the polynomials
$\mathbb C[u,v]\cap [g,h,J]$.
 It is usually a principal ideal and the generating polynomial $R(u,v)$ of this ideal
describe the discriminant locus of complexified variety.
We define the branching locus $\Sigma_{\mathbb R}$  by the intersection $\Sigma_{\mathbb C}\cap \mathbb R^2$. Take a point $w\in \Sigma_{\mathbb R}$.
It is not always true that a point $w\in \Sigma_{\mathbb R}$ is a branching point of $p:C\to \mathbb R^2$.
 It might come from the branching on the complex point of $C(\mathbb C)$ outside of $C$.
 That is $\Sigma\subset \Sigma_{\mathbb R}$ but the equality does not hold in general. See Example 2 below. Also it might have some  point $\eta_0$ such that $L_{\eta_0}\cap C$ contains a 1-dimensional intersection.  See Remark \ref{remark1}.
 
There are some cases for which this branching loci are comparatively simple.
Suppose that $f$ is a join type polynomial of $z_1^{q+r}{\bar z_1}^r$ and strongly mixed homogeneous convenient polynomial 
$K(z_2,z_3,\bar z_2,\bar z_3)$ of two variables $z_2$
and $ z_3$. 
Then affine equation takes the form
$f(\bfz,\bar \bfz)=z^{q+r}{\bar z}^r+k(w,\bar w)=0$ with respect to the affine coordinates $z=z_1/z_3$ and $w=z_2/z_3$.
 By the non-degeneracy assumption, 
the roots of $k(w,\bar w)=0$ are all simple. Then the  branching locus $\Sigma$ is nothing but the set of those roots and over any of these roots, 
the projection is locally equivalent to $q$ cyclic coverings
$z^{q+r}{\bar z}^r-w=0$ or $z^{q+r}{\bar z}^r-\bar w=0$ depending the sign of the root.
\begin{figure}[htb]
\setlength{\unitlength}{1bp}
\begin{picture}(600,200)(-100,-600)
{\includegraphics [width=8cm,height=6cm]
{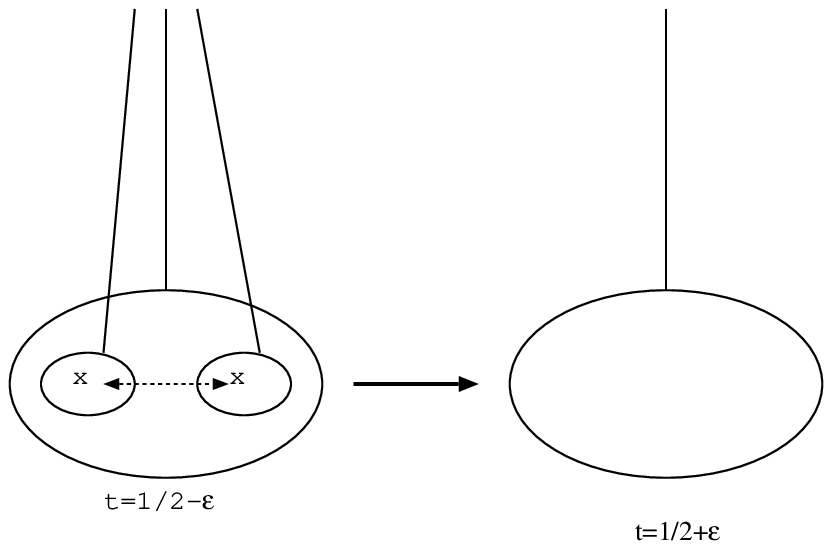}}
\put(-200,-520){$\omega_i$}
\put(-180,-560){$\xi_{2i}$}
\put(-220,-560){$\xi_{2i-1}$}
\put(-260,-560){$D_i$}
\put(-120,-520){$\omega_i$}
\put(-120,-560){$\delta_i$}
\put(-45,-520){$\omega_i$}
\end{picture}
\vspace{0cm}
\caption{Vanishing loops}\label{vanishing}
\end{figure}

However for a generic mixed polynomial, $\Sigma_{\mathbb R}$ and $\Sigma$ are  much more complicated. Usually
they have real dimension 1 components and also it can have isolated points. We assume that for each $\eta$, $L_{\eta}\cap C$ is a finite point. We define $\gamma(\eta)$ by the cardinality of $L_{\eta}\cap C$.　We divide $\{\mathbb C \setminus \Sigma,\Sigma-S(\Sigma),S(\Sigma)\}$ by $\gamma$-values 
where $S(\Sigma)$ is the singular locus of $\Sigma$ and let $\mathcal D$ be the corresponding division.
We call $\mathcal D$ {\em the $\gamma$-division} of the parameter space. 2-dimensional
 connected component $V\in \mathcal D$ (respectively 1-dimensional $L$, 0-dimensional $P$) is called a {\em region }(resp.  an {\em edge},  a {\em  vertex}).
A region $V$ is called {\em  regular} if
 the inclusion map $V\subset \bar V$ is a homotopy equivalence. An edge $L$ is called {\em regular} if 
 there exists exactly two regions, say $V_1,V_2$ 
 whose boundaries contain $L$ and  $\gamma(L)=(\gamma(V_1)+\gamma(V_2))/2$.
 A vertex $P$ is called {\em regular} if there exists at most two regions which contains $P$ in its boundary.

  For an regular edge $M\in \mathcal D$,
suppose that two regions $S_1,S_2$ are touching each other along $M$ and suppose that $\gamma(S_1)>\gamma(S_2)$. Take 
a  point $a\in M$ and a small transversal path $\si:[0,1]\to \mathbb R^2=\mathbb C$ so that $\si(t)\in S_1$ for $t<1/2$, $\si(1/2)=a$
and  $\si(t)\in S_2$ for $t>1/2$. Let  $\gamma:=(\gamma(S_1)-\gamma(S_2))/2$.
Then  for a sufficiently small $\eps>0$ and $1/2-\eps\le \forall t<1/2$,
$p^{-1}(\si(t))$ consists of $\gamma(S_1) $ points, say $\xi_1(t),\dots, \xi_{\gamma(S_1)}(t)$ and among them there exist
$\gamma$ pairs of points $\{\xi_{2i-1}(t),\xi_{2i}(t)\},\,i=1,\dots, \gamma$ and we can choose  continuous family of disjoint $\gamma$ disks $D_i(t),i=1,\dots,\gamma$ in
the pencil line $p^{-1}(\si(t))=\mathbb R^2$  
and contain only the corresponding pair of roots so that when $t$ goes to $1/2$, 
two roots approach each other in the disk $D_i(t)$ and collapse to $\de_i\in L_a\cap M$, a double point  and then they disappear for $t>1/2$. 
These pairs of roots  $\{\xi_{2i-1}(t),\xi_{2i}(t)\}$ as roots of a polynomial equation  $f(z,\si(t))=0$ has positive and negative sign. 
Take a base point $b=[1:0:0]$ of the fundamental group at the base
point of the pencil.  Consider a loop
$\omega_i\in \pi_1(L_{\si(1/2-\eps)}-C,b)$ represented by the boundary loop of $D_i(1/2-\eps)$, connected to the base point by a path outside of $L_{\si(1/2-\eps)}\setminus\cup_{i=1}^k D_{i}( \si(1/2-\eps))$.
Then we get  the following relation for $t:1/2-\epsilon\to 1/2+\epsilon$
\[
\omega_i=e,\quad i=1,\dots, \gamma.
\]
Take elements $\xi_{2i-1},\xi_{2i}$ as in Figure \ref{vanishing}. Then this implies that
\[
\omega_i=\xi_{2i-1}\xi_{2i}=e\,\,\text{or equivalently}\,\,\xi_{2i-1}=\xi_{2i}^{-1},\,i=1,\dots,\gamma.
\]
We call these relations {\bf vanishing monodromy relations}.
\subsection{Example 1}
Now we present several examples which are not either simplicial or of join type but the complement has an abelian fundamental group.
\subsubsection{Example 1-1}
Consider the following mixed curve of polar degree 1
\[C_t:\quad F(\bfz,\bar\bfz):= {z_1}^{2}{ \bar z}_1+{z_2}^{2}
 \bar{z}_2 +z_3^2\bar{z}_3 +t\,{z_1}^{2}z_2{\bar z}_3=0
\]
with  $t\in \mathbb C$
and let $C_t$ be the corresponding projective curve. 
Let $M_t=F\inv(1)$ be the corresponding Milnor fiber.
Then $C_0$ is a mixed Brieskorn type and  isotopic to the standard line $z_1+z_2+z_3=0$, namely a sphere $S^2$ (see \cite{OkaBrieskorn})
and $M_0$ is diffeomorphic to the plane $\mathbf C$. 
This is true for any small $t$.　Observe  that $\{z_3=0\}\cap C_t=\{[1:-1:0]\}$. 

We are interested in  $C:=C(-4):\, z_1^2\bar z_1+z_2^2\bar z_2+z_3^2\bar z_3-4 \,z_1z_2\bar z_3$.
We use the notation $M_{-4}=M$ for simplicity.
Take the affine coordinate
$z=z_1/z_3,\, w=z_2/z_3$.
Then the affine equation is given as 
\[
C:\quad z^2\bar z+w^2 \bar w+1-4\, z w = 0.
\]
 To compute the Euler characteristic $\chi(C)$ and the fundamental group \nl $\pi_1(\mathbf P^2\setminus C)$,
we consider the pencil $L_\eta:=\{w=\eta\},\,\eta=u+vi\in \mathbb C$. The branching locus $\Sigma_{\mathbb R}$ is given by $R(u,v)=0$
where
\begin{multline*}
R(u,v)=
27+11642\,{v}^{2}{u}^{4}-2640\,{u}^{7}{v}^{2}+405\,{u}^{4}{v}^{8}+162
\,{u}^{10}{v}^{2}+16438\,{v}^{6}\\
-6736\,{v}^{6}{u}^{3}-350\,{u}^{6}+405
\,{v}^{4}{u}^{8}+162\,{u}^{2}{v}^{10}+27\,{v}^{12}+540\,{v}^{6}{u}^{6}
\\+28430\,{u}^{2}{v}^{4}
-148\,{u}^{3}-2196\,u{v}^{2}
-7032\,{u}^{5}{v}^{4
}+27\,{u}^{12}-2196\,u{v}^{8}-148\,{u}^{9}.
\end{multline*}
See Appendix 1 (\S \ref{Appendix1}) for the practical computation of  $R(u,v)$.
Its graph of the zero locus set $R=0$ is  given as Figure  \ref{NJ}.
Let $A$ be the  bounded region of $\mathbb C\setminus \Sigma_{\mathbb R}$ and let $U$ be the complement
 $\mathbb P^1\setminus\bar A$. There are four singular points $V_i,i=1,\dots,4$ of the boundary of $\bar A$. 
 Actually $-1$ is an isolated point of $\Sigma_{\mathbb R}$ but $L_\eta\cap C$ consists of one simple point and it does not give any
  branching of the projection $p:C\to \mathbb C$. Thus $-1\notin \Sigma$.

\begin{figure}[htb]
\setlength{\unitlength}{1bp}
\includegraphics[width=10cm,height=7cm]{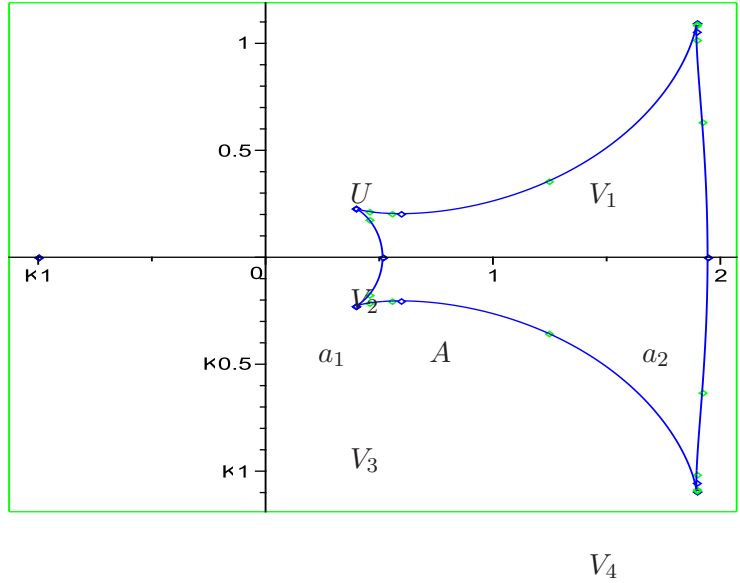}
\put(-120,60){$A$}
\put(-162,60){$a_1$}
\put(-40,60){$a_2$}
\put(-60,120){$V_1$}
\put(-150,120){$U$}
\put(-150,20){$V_3$}
\put(-150,80){$V_2$}
\put(-60,-20){$V_4$}
\vspace{1.5cm}
\caption{Graph of $R=0$, Example 1-1}\label{NJ}
\end{figure}
As the polar degree is 1, the number of intersection $L_\eta\cap C$ counted with sign is always 1.
Observe that $L_\eta\cap C$ consist of 3 simple roots of $f(z,\eta)=0$ for any
$\eta\in A$. 
Observe further that  over any point $\eta$ of the complement $U$ of $\bar A$, $L_\eta\cap C$ has a unique simple root, i.e. $\gamma(U)=1$ and 
$\gamma(A)=3$. For any smooth  boundary point $\eta$ of $\partial {\bar A}$, $L_\eta\cap C$ has two points, one simple and one double point. (Strictly speaking, there does not exists of the notion of multiplicity in the mixed roots. See \cite{MixIntersection}. Here we use the terminology ``double root'' in the sense that it is a limit of two simple roots).
As for four singular points, we have  $\gamma(V_i)=1$, $i=1,\dots,4$.
Let $\overline{a_1a_2}$ be the line segment cut by $A\cap \{v=0\}$ where  $a_1\approx 0.51, a_2\approx 1.94$.
For any $a_1<\eta<a_2$, $L_\eta\cap C$ has three simple points which are all real. This can be observed by the graph of $f=0$ restricted on  the real plane section $(w,z)\in \mathbb R^2$ ( Figure \ref{NJ-real}).
Consider the limit of 
$L_\eta\cap C$ when $\eta$ goes to $a_1$ or $a_2$ along the real line segment $\overline{a_1a_2}$. 
There are two real positive roots and one real negative root and at the both end, two positive roots collapse to a double point,
which is clear from Figure \ref{NJ-real}. 

\begin{figure}[htb]
\setlength{\unitlength}{1bp}
\includegraphics[width=10cm,height=7cm]{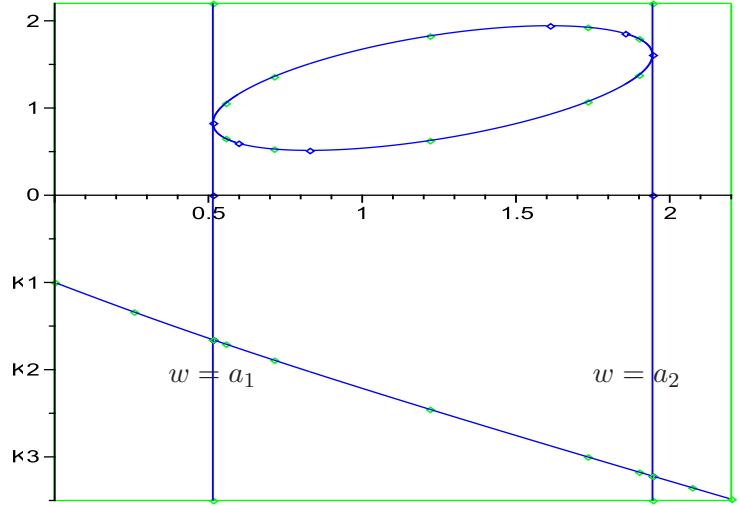}
\put(-220,50){$w=a_1$}
\put(-60,50){$w=a_2$}
\vspace{1.5cm}
\caption{Graph of $f=0$, Example 1-1}\label{NJ-real}
\end{figure}

Using these data, we can compute the Euler characteristic as 
\begin{eqnarray*}
\chi(C)&=&\chi(p^{-1}(\bar A))+\chi(p^{-1}(U))\\
&=&-1 +1=0.
\end{eqnarray*}
This  implies $C$ is a  torus and $\chi(\mathbb P^2-C)=3-0=3$. 
We  claim
\begin{Proposition}
\begin{enumerate}
\item $\pi_1(\mathbb P^2-C)=\{e\},\quad \pi_1(M)=\{e\}$.
\item$\chi(M)=3,\,H_1(M)=0,\,H_2(M)=2$.
\end{enumerate}
\end{Proposition}
\begin{proof}
We  first compute the fundamental group. Put $b_0=1$ and we take $L_{b_0}$ as a fixed regular pencil line.
Then $L_{b_0}\cap C=\{x_1,x_2,x_3\}$ where
\[x_1<0 <x_2<x_3.\]
See Figure \ref{NJ-real}.
It is not hard to see that 
$\pi_1(L_{b_0}\setminus C\cap L_{b_0})\to \pi_1(\mathbb P^2\setminus C)$ is surjective.
See \S \ref{surjective} for an explanation in detail.
Take generators $\xi_1,\xi_2,\xi_3$ of $\pi_1(L_{b_0}\setminus C\cap L_{b_0})$ as in Figure \ref{generatorsExample1}.
They are oriented counterclockwise.
First, as a vanishing relation at infinity, they satisfy the relation
\begin{eqnarray}
\xi_1\xi_2\xi_3=e.
\end{eqnarray}

\begin{figure}[htb]
\setlength{\unitlength}{1bp}
\begin{picture}(600,200)(-150,-620)
{\includegraphics [width=6cm,height=4cm]
{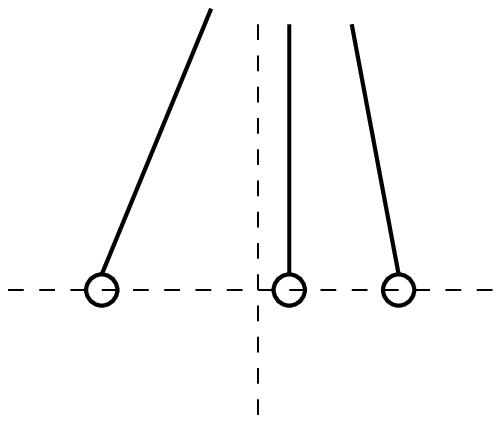}}
\put(-140,-520){$\xi_3$}
\put(-123,-550){$\cdot$}
\put(-170,-520){$\xi_2$}
\put(-159,-550){$\cdot$}
\put(-220,-520){$\xi_1$}
\put(-222,-550){$\cdot$}
\end{picture}
\vspace{-2cm}
\caption{Generators of $\pi_1(L_{b_0}-C\cap L_{b_0})$}\label{generatorsExample1}
\end{figure}
When $\eta$ moves on the interval $[a_1,a_2]$ from $\eta=b_0$ to $a_1$ or $a_2$, we see that two positive roots collapse to a point and disappear for $\eta<a_1$ or $\eta>a_2$. Thus as a vanishing relation, we get
\[
\xi_2=\xi_3^{-1}.
\]
Now we consider the movement from $\eta=1$ along the vertical line to $\eta=1+v_0i$ where $(1,v_0)\in \partial A$ and $v_0\approx 0.26$. The generators are deformed as in Figure \ref{moveExample1}.
Thus as a vanishing monodromy relation, we get $(\xi_2^{-1}\xi_1\xi_2)\xi_3=e$.
Thus combining the above relations, we get
\[\xi_1=\xi_3,\, \xi_2=\xi_1,\,\text{and}\,\,\xi_1=e.\]
Namely we conclude that $\pi_1(\mathbb P^2-C)=\{e\}$.
\end{proof}
\begin{figure}[htb]
\setlength{\unitlength}{1bp}
\begin{picture}(600,200)(-130,-620)
{\includegraphics [width=8cm,height=6cm]
{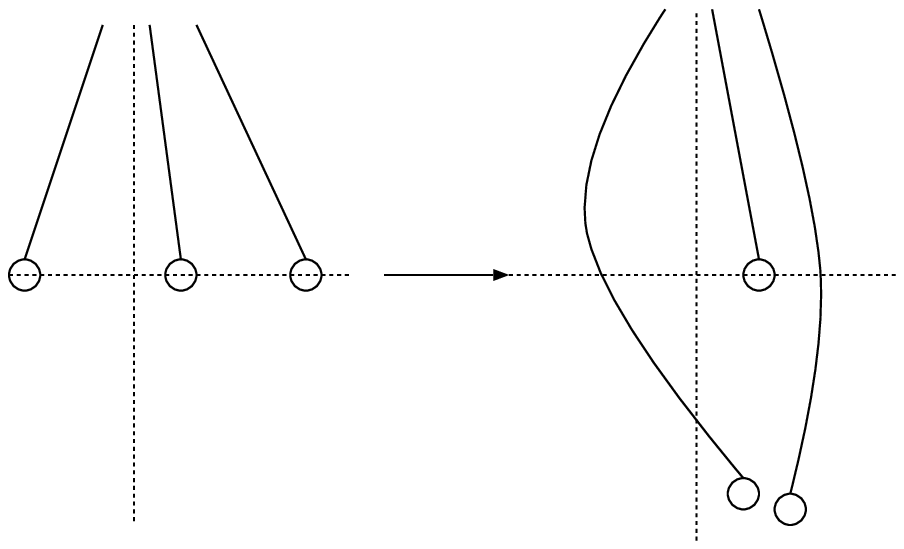}}
\put(-190,-520){$\xi_3$}
\put(-181,-543){$\cdot$}
\put(-220,-520){$\xi_2$}
\put(-212,-543){$\cdot$}
\put(-260,-520){$\xi_1$}
\put(-252,-543){$\cdot$}
\put(-50,-520){$\xi_3$}
\put(-80,-520){$\xi_2$}
\put(-70,-612){$\cdot$}
\put(-68,-543){$\cdot$}
\put(-60,-616){$\cdot$}
\put(-110,-520){$\xi_1$}
\put(-230,-560){$\eta=1$}
\put(-140,-560){$\eta=1+(v_0-\eps)i$}
\end{picture}
\vspace{0cm}
\caption{movement on $\eta=1+si$, Example 1}\label{moveExample1}
\end{figure}
\begin{Remark}\label{remark1}
It can be observed that the set  $\Gamma:=\{t=t_1+t_2i\in \mathbb C\,|\, C_t:\,\text{singular}\}$ 
is a real one dimensional 
semi-algebraic set and  the complement $\mathbb C\setminus\Gamma$ has two connected component in this case.
 The bounded region contain $0$
and for any $t$ in this region, $C_t$
is isotopic  to $C_0$ and it is a  rational sphere. $\Gamma$ is calculated by Groebner basis calculation.
In our case,
 we found that $\Gamma$ is defined by  
\[  {{\it t_1}}^{4}-6\,{{\it t_1}}^{2}+8\,{\it t_1}-3+2\,{{
\it t_2}}^{2}{{\it t_1}}^{2}-6\,{{\it t_2}}^{2}+{{\it t_2}}^{4} =0. 
\]
Certainly $C_{-4}$ is in the outside unbounded region. We may choose another one $C_{\root 3\of 3}$  which must be isotopic to $C_{-4}$ but the branching locus is very different and defined by $R=0$
and its graph is given by Figure \ref{NJ1-bis2}.
\begin{figure}[htb]
\setlength{\unitlength}{1bp}
\includegraphics[width=10cm,height=7cm]{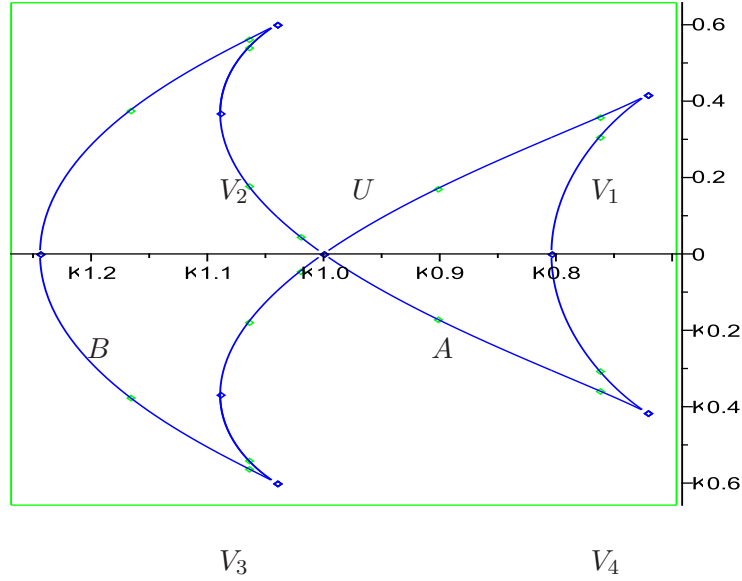}
\put(-120,60){$A$}
\put(-250,60){$B$}
\put(-60,120){$V_1$}
\put(-150,120){$U$}
\put(-200,-20){$V_3$}
\put(-200,120){$V_2$}
\put(-60,-20){$V_4$}
\vspace{1.5cm}
\caption{Graph of $R=0$, Remark1}\label{NJ1-bis2}
\end{figure}
In this example, $\ga(A)=\ga(B)=3$ but the point $(u,v)=(-1,0)$ is special as $L_{-1}\cap C$ has one simple point and one 1-dimensional
component which is defined by $|z|=\root 6\of 3$. Thus the geometry of the pencil is more complicated and it takes more careful consideration to compute the fundamental group.
\begin{multline*}
R(u,v):=27+540\,{v}^{6}{u}^{6}+405\,{v}^{4}{u}^{8}+162\,{u}^{2}{v}^{10}
+120\,{
u}^{9}+162\,{u}^{10}{v}^{2}+654\,{u}^{4}{v}^{2}\\
+120\,{u}^{3}+216\,u{v}
^{2}+1008\,{u}^{5}{v}^{4}+216\,u{v}^{8}+27\,{u}^{12}+405\,{v}^{8}{u}^{
4}+768\,{v}^{6}{u}^{3}\\
+576\,{u}^{7}{v}^{2}+90\,{v}^{6}+558\,{u}^{2}{v}
^{4}+186\,{u}^{6}+27\,{v}^{12}.
\end{multline*}

\end{Remark}
\subsubsection{Example 1-2}
We consider another example with polar degree 1 and radial degree 3.
Let $F(\mathbf z,\bar{\mathbf z}):=z_1^2\bar z_1+z_2^2\bar z_2+z_3^2\bar z_3-4z_2z_3\bar z_3-2 z_3^2\bar z_1$.
Taking the affine chart $\{z_3\ne 0\}$ and coordinates $z=z_1/z_3, w=z_2/z_3$, the affine equation is
$f(z,w)=z^2\bar z+w^2\bar w+1-4w-2\bar z$.
Consider the pencil $L_\eta:=\{w=\eta\},\eta\in \mathbb C$. Putting $w=u+vi$, the branching locus is described by $R=0$
where the explicit form is given in Appendix 2(\S \ref{eq-R})
 to show that  the equation of $R$ grows exponentially by the number of monomials and degree.
However the graph of $R=0$ is not so complicated and it is given in Figure \ref{NJ-bis-singular4}.
We observe that $\ga(W_i)=1,\,i=1,2,3$ and $\ga(T)=3$ where $T$ is the complement of $\overline{W}_1\cup \overline{ W}_2\cup \overline{ W}_3$.
There are two singular points of the boundary of $T$, $V_1,V_2$ and $\ga(V_i)=1$ and the other boundary points have 2 roots.
Let $a_1,\dots,a_6$ be  real roots of $R(u,0)=0$ and we assume that $a_1<a_2<\dots<a_6$. Note that
$a_1\approx -2.22$ and $a_2\approx -1.98$. See figure \ref{NJ-bis-singular4}. In the Figure, the horizontal line is $w$-coordinate.
Take a base line $L_{b_0}$ with $b_0=a_2-\eps$, $0<\eps\ll1$. See the graph of $f=0$ on $\mathbb R^2$ (Figure \ref{NJ-bis-singular4-real}).
Two vertical lines are $w=a_1$ and $w=a_2$. We take generators $\xi_1,\xi_2,\xi_3$ of $\pi_1(L_{b_0}\setminus C)$
as the left side of
 Figure \ref{generatorsExample1}.
 Considering the movement of $\eta=b_0$ to $\eta=a_1$ and from $\eta=b_0$ to $\eta=a_2$, we get the vanishing monodromy relations
 \begin{eqnarray}
 \xi_1\xi_2=e,\quad \xi_2\xi_3=e.
 \end{eqnarray}
 This is also clear from Figure \ref{NJ-bis-singular4-real}.
 Thus $\pi_1(\mathbb P^2\setminus C)$ is abelian and therefore we conclude that $\pi_1(\mathbb P^2\setminus C)=H_1(\mathbb P^2-C)$ is trivial.
 The Euler number is computed as 
 \[
 \chi(C)=\chi(W_1)+\chi(W_2)+\chi(W_3)+\chi(T)-\chi(\partial T)=1+1+1-3-2=-2.
 \]
 Thus the genus of  $C$ is $2$.
 \begin{figure}[htb]
\setlength{\unitlength}{1bp}
\includegraphics[width=10cm,height=7cm]{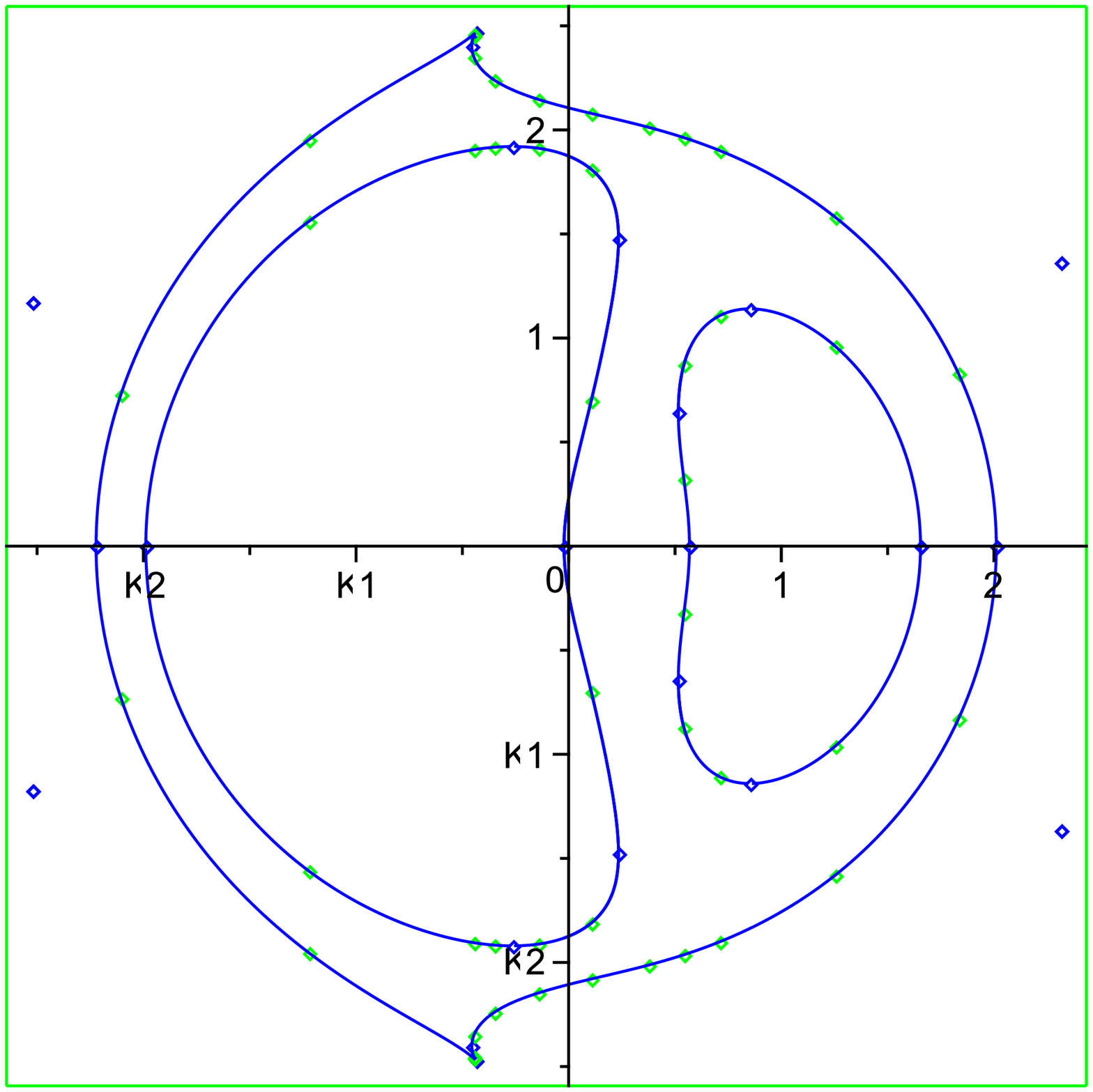}
\put(-280,60){$a_1$}
\put(-252,60){$a_2$}
\put(-160,60){$a_3$}
\put(-128,60){$a_4$}
\put(-70,60){$a_5$}
\put(-50,60){$a_6$}
\put(-45,120){$W_3$}
\put(-165,150){$V_1$}
\put(-165,-40){$V_2$}
\put(-110,80){$W_1$}
\put(-130,80){$T$}
\put(-200,80){$W_2$}

\caption{Graph of $R=0$, Example 1-2}\label{NJ-bis-singular4}
\end{figure}

\begin{figure}[htb]
\setlength{\unitlength}{1bp}
\includegraphics[width=8cm,height=5cm]{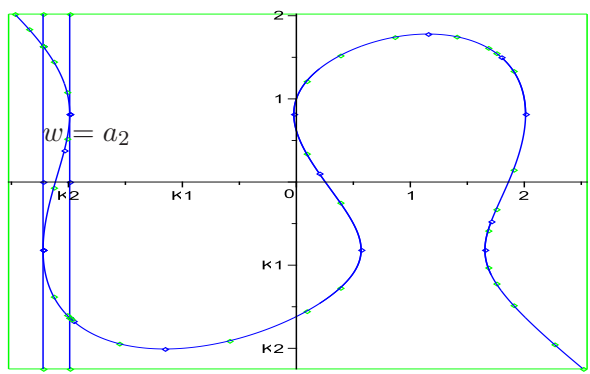}
\put(-240,-20){$w=a_1$}
\put(-210,90){$w=a_2$}

\caption{Graph of $f=0$ with $z,w\in \mathbb R$, Example 1-2 }\label{NJ-bis-singular4-real}
\end{figure}
\subsection{Example 2}
Consider the next mixed curve of polar degree 2 and radial degree 4.
\[
C_t\quad F(\mathbf z,\bar{\mathbf z}) :=\bar z_1z_1^3+z_2^3\bar z_2+z_3^3\bar z_3+t\, z_1^2 \bar z_2 z_3
\]
For $t$ small, $C_t$ is isotopic to the  conic $z_1^2+z_2^2+z_3^2=0$ in $\mathbb P^2$ and a rational sphere (\cite{OkaBrieskorn}).
We take $t=-4$ and put $C=C_{-4}$ and $M=M_{-4}$ the Milnor fiber.
The branching locus is defined by $R=R_1R_2=0$ where
\begin{multline*}
R_1:=1+{u}^{8}+6\,{v}^{4}{u}^{4}+2\,{u}^{4}-2\,{v}^{4}+4\,{v}^{6}{u}^{2}+4
\,{u}^{6}{v}^{2}+{v}^{8},
\\
R_2:=1-12\,{u}^{4}+124\,{v}^{4}-320\,{u}^{2}{v}^{2}+8\,{u}^{14}{v}^{2}+
12580\,{v}^{4}{u}^{4}+12936\,{v}^{6}{u}^{2}\\
+3464\,{u}^{6}{v}^{2}+70\,{
v}^{8}{u}^{8}-1228\,{u}^{8}{v}^{4}+56\,{v}^{10}{u}^{6}-1472\,{v}^{6}{u
}^{6}+56\,{v}^{6}{u}^{10}\\-548\,{v}^{8}{u}^{4}
-26\,{u}^{8}+3846\,{v}^{8
}-12\,{u}^{12}+124\,{v}^{12}+{u}^{16}+28\,{v}^{12}{u}^{4}\\+8\,{v}^{14}{
u}^{2}+28\,{u}^{12}{v}^{4}
-368\,{u}^{10}{v}^{2}+176\,{v}^{10}{u}^{2}+{
v}^{16}.
\end{multline*}
We claim that 
\begin{Proposition}
\begin{enumerate}
\item $\pi_1(\mathbb P^2- C)\cong \mathbb Z/2\mathbb Z$.
\item $\chi(C)=-2$. The genus of $C$ is 2.
\end{enumerate}
\end{Proposition}

\begin{figure}[htb]
\setlength{\unitlength}{1bp}
\includegraphics[width=10cm,height=7cm]{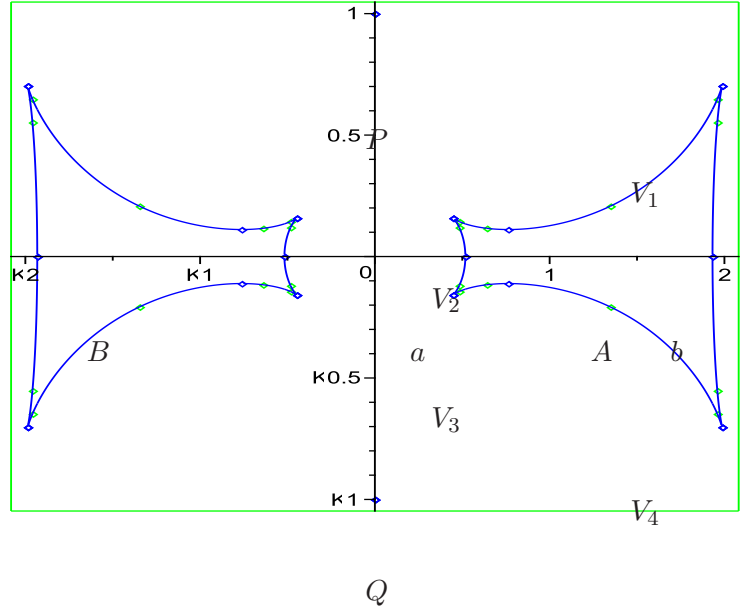}
\put(-60,60){$A$}
\put(-30,60){$b$}
\put(-250,60){$B$}
\put(-45,120){$V_1$}
\put(-145,140){$P$}
\put(-145,-30){$Q$}
\put(-120,80){$V_2$}
\put(-128,60){$a$}
\put(-120,35){$V_3$}
\put(-45,0){$V_4$}
\vspace{1.5cm}
\caption{Graph of $R=0$, Example 2}\label{NJ2-bis2}
\end{figure}
\begin{figure}[htb]
\setlength{\unitlength}{1bp}

\begin{picture}(600,200)(-150,-670)
{\includegraphics [width=8cm,height=6cm]
{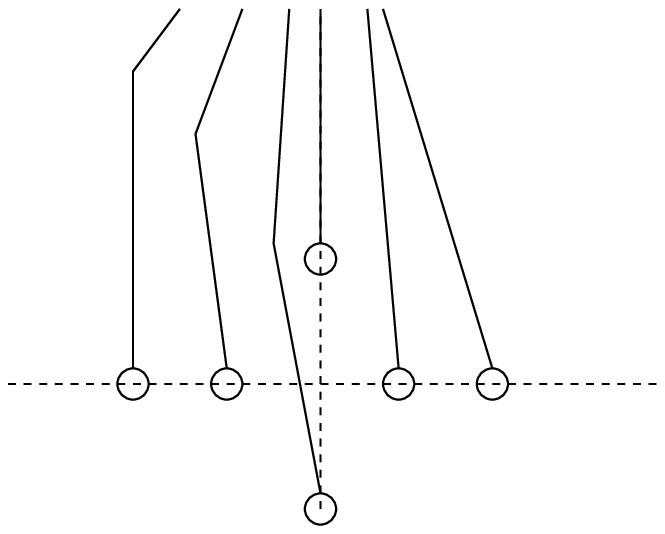}}
\put(-140,-620){$\xi_6$}
\put(-145,-632){$\cdot$}
\put(-175,-600){$\xi_5$}
\put(-178,-632){$\cdot$}
\put(-200,-570){$\xi_4$}
\put(-205,-592){$\cdot$}
\put(-200,-660){$\xi_3$}
\put(-205,-672){$\cdot$}
\put(-235,-600){$\xi_2$}
\put(-237,-632){$\cdot$}
\put(-265,-600){$\xi_1$}
\put(-268,-632){$\cdot$}
\end{picture}
\vspace{0cm}
\caption{Generators of $\pi_1(L_{1}-C\cap L_{1})\, (\eta=1)$}\label{generatorsExample2}
\end{figure}
\begin{proof} The locus $R_1=0$ gives two isolated points $P=(0,1), \,Q=(0,-1)$.
Note that the affine equation of $C$ is defined by $f(z,w)=z^3\bar z+w^3\bar w+1-4 z^2 \bar w=0$. Recall that
the highest degree part of $f$  as a polynomial of $z$ is $z^3\bar z$, and therefore the number of roots counted with sign is two by 
\cite{MixIntersection}.
We also observe that$f(z,w)=0\iff f(-z,w)=0$.  Thus roots are symmetric with respect to the origin in $z$-coordinates.
$R$ is symmetric with respect to $v$-axis but the region $B$ does not give any branching. It comes from the complex part of the curve. Thus $\ga(B)=2$. 
Also we observe that $\ga(A)=6$ and $\ga(\partial A)=4$ except 4 singular points $V_1,\dots, V_4$ where $L_\eta\cap C$ has 2 multiple points.
The complement region
$E:=\mathbb P^1\setminus(\bar A\cup\{P,Q\})$ has 2 simple points for any fiber $L_\eta\cap C$ with $\eta\in E$. $\ga(P)=\ga(Q)=1$.
Take generators of $\pi_1(L_1-C)$, $\xi_i,i=1,\dots, 6$ as in Figure \ref{generatorsExample2}. Observe that 
$f(z,w)=0$ implies $f(-z,w)=0$. Thus the roots are always paired by $z, -z$ for a fixed $w$.
Put $\bar A\cap \{v=0\}=\{a,b\}$ with
$a\approx 0.51$ and $b\approx 1.93$.
First we consider the movement $\eta=1\to a$. 
Consider the graph of $f_r:=z^4+w^4+1-4z^2w$ (Figure\ref{NJ2-bis2-real}) where $f_r$ is the restriction of $f$ to $\mathbb R^2$. 
This says that on $[a,b]$, $L_\eta\cap C$ has exactly four real roots, which are symmetric with respect the origin and at $\eta=a$, they collapse to two double roots.
Put  $f_i$  be  the restriction of $f(iz,w)$ to $(z,w)\in \mathbb R^2$
and its graph.  See Figure \ref{NJ2-bis2-imaginary}.
Using the real graph of $f_i:=-z^4+w^4+1+4z^2w$, we see also that there are exactly two purely imaginary root of $f(z,\eta)=0$ for any $w\in \mathbb R$.
The above observation says that
\begin{eqnarray}\label{rel1}
\xi_1\xi_2=e,\quad \xi_5\xi_6=e.
\end{eqnarray}
(The Figure \ref{NJ2-bis2-real} shows that we get the same degeneration for $\eta\to b$.)
Then we  consider  movement of  the line $L_\eta$ further to the left until $\eta=0$. Then we move $L_\eta$ along the imaginary axis to $\eta=i$ which is a root of multiplicity 2 (P in the graph). Note that the monodoromy along $|w-i|=\eps$ is topologically half turn of two roots.
Thus we get 
\begin{eqnarray}\label{rel2}
\xi_3=\xi_4.
\end{eqnarray}
\begin{figure}[htb]
\setlength{\unitlength}{1bp}
\begin{picture}(300,200)(-50,-20)
{\includegraphics[width=8cm,height=5cm]{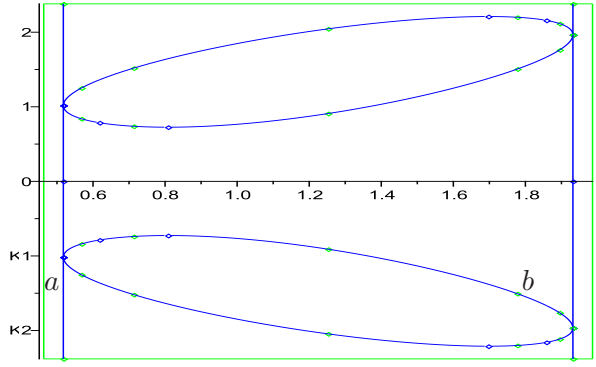}}
\put(-210,30){$a$}
\put(-30,30){$b$}
\end{picture}
\vspace{1.5cm}
\caption{Graph of $f=0$, Example 2}\label{NJ2-bis2-real}
\end{figure}
\begin{figure}[htb]
\setlength{\unitlength}{1bp}
\begin{picture}(300,200)(-50,-20)
{\includegraphics[width=8cm,height=5cm]{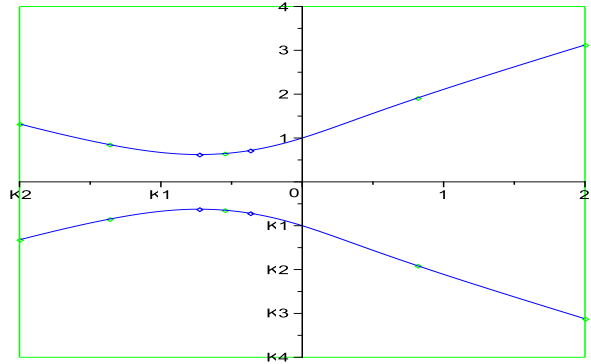}}
\end{picture}
\vspace{1.5cm}
\caption{Graph of $f_i(z,w)=0$, Example 2}\label{NJ2-bis2-imaginary}
\end{figure}

Now we will see the vanishing relation along the vertical line for $\eta=1\to 1+c_0i\dots$ where $c_0$ is the positive root of $R_2(1,v)=0$. The root of 
$f(z,w)=0$ with $w=1+ci$ is given as
\[\pm P_1,\pm P_2,\quad P_1\approx 1.57-1.21i,P_2\approx 0.76+0i
\] where $\pm P_1$ are double roots. Recall that 
$f(z,1)=0$ has roots
\[\pm  Q_1, \pm Q_2,\pm Q_3,\quad
Q_1\approx 2.11i,\,Q_2\approx 0.76, Q_3\approx 1.84.
\]

The movement of generators during the above movement are described in Figure \ref{Move-eps.eps}.
The doted loops show the situation in $\eta=1+(c_0-\eps)i$ wit $0<\eps\ll 1$. They are denoted as $\xi_1',\dots,\xi_6'$.
In this movement, $\xi_2,\xi_5$ does not move much. Other generators are deformed  as indicated with arrows.
At $\eta=1+c_0i$, $\xi_1',\xi_4'$ and $\xi_3',\xi_6'$ collapse respectively.
Thus  $\xi_1=\xi_1',\xi_4=\xi_4'$  and $\xi_3=\xi_3',\xi_6=\xi_6'$ and we get vanishing relations which is written as
\begin{eqnarray}\label{rel3}
\xi_1\xi_4=e,\quad (\xi_4\xi_5)^{-1}\xi_3(\xi_4\xi_5) \xi_6=e.
\end{eqnarray}
Using (\ref{rel1}), (\ref{rel2})and (\ref{rel3}), we conclude that
\[
\xi_2=\xi_1^{-1},\,\xi_3=\xi_1,\,\xi_4=\xi_1^{-1},\xi_5=\xi_1,\,\xi_1^2=e.\]
That is, $\pi_1(\mathbb P^2-C)\cong\mathbb Z/2\mathbb Z\cong H_1(\mathbb P^2-C)$.
\end{proof}
\begin{figure}[htb]
\setlength{\unitlength}{1bp}

\begin{picture}(600,200)(-150,-670)
{\includegraphics [width=8cm,height=6cm]
{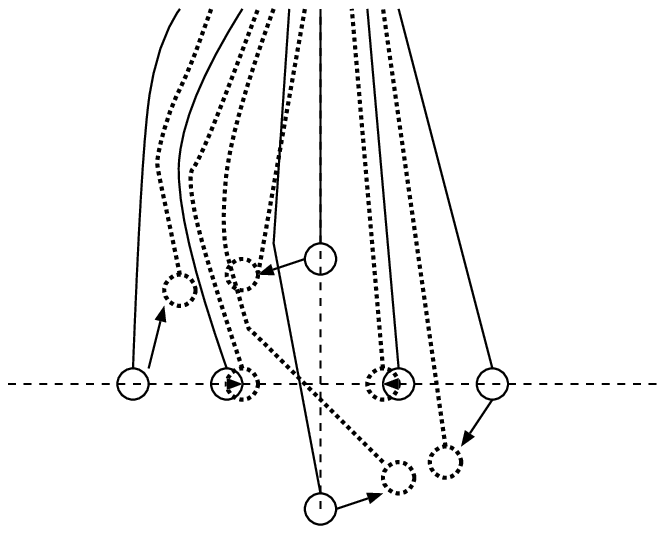}}
\put(-140,-620){$\xi_6$}
\put(-160,-630){{$\xi_6'$}}
\put(-175,-600){$\xi_5$}
\put(-185,-600){$\xi_5'$}
\put(-200,-570){$\xi_4$}
\put(-220,-580){$\xi_4'$}
\put(-200,-660){$\xi_3$}
\put(-180,-670){$\xi_3'$}
\put(-235,-640){$\xi_2$}
\put(-225,-640){$\xi_2'$}
\put(-280,-600){$\xi_1$}
\put(-255,-610){$\xi_1'$}
\end{picture}
\vspace{0cm}
\caption{Movement of generators, Example 2}\label{Move-eps.eps}
\end{figure}
\newpage
\subsection{ Surjectivity}\label{surjective} Assume that $f(z,\bar z,w,\bar w)=0$ is the affine equation of a non-singular  mixed curve $C$ of polar degree $q$ and radial degree $q+2r$. We assume that 
$f$ is monic in the sense that it has the monomial $z^{q+r}{\bar z}^r$ with a non-zero coefficient.
Consider the pencil line $L_\eta=\{w=\eta\},\eta\in \mathbb C$
and we consider the $\gamma$-division $\mathcal D$ of $\mathbb C$ (the parameter space) by the value of $\gamma(\eta)$ using  the graph of $R$. We assume that every region, edges and vertices are regular.
We assume also that the base point $b$ of the pencil is not on $C$.
Let $G=\{\gamma(\eta)\,|\, \eta\in \mathbb C\}\subset\mathbb N$
the possible number of roots of $f(z,\bar z,\eta,\bar \eta)=0$ and $\gamma_{max}$ be the maximum of $G$. We assume the following two conditions.
\newline
(1) The set $U_{max}:=\{\eta\,|\, \gamma(\eta)=\gamma_{max}\}$ is connected and it is  a region. 
\newline
(2) Take  a region $U$   of $\mathcal D$ with $\gamma(U)<\gamma_{max}$. 
Put
$\partial_{+}U=\{q\in \partial U\,|\, \gamma(q)\ge \gamma(U)\}$. Then $\partial_+ U$ is connected.

Note that the above condition is satisfied in Example 1-1, Example 1-2, Example 2.
Let $B$ be the complement of the union of regions of $\mathcal D$, i.e. $B$ is the union of the edges and vertices.
We fix a generic line $L_{\eta_0}$ with $\eta_0\in U_{max}$ and a base point $b\in L_{\eta_0}\setminus C$.
Let $\sigma:(I,\{0,1\})\to (\mathbb P^2\setminus C,b)$ be a loop. We may assume that $\{t\,|\, \sigma(t)\in B\}$ is finite.
Let $\alpha:=\min\{\gamma(\si(t))\,|\,t\in I\}$ and we may assume that $\alpha$ is taken in a region $V$ of $\mathcal D$.
Put $\mathcal D_{\beta}$ be the union of $\overline U$ with $\gamma(U)\ge \beta$. The we assert:
\begin{Assertion}
$\sigma$ is homotopic in $\mathbb P^2\setminus C$ to  a loop $\hat\sigma$ in the pencil line $L_{\eta_0}\setminus C$.
 \end{Assertion}
\begin{proof}
We may assume that $\pi\circ \sigma$ intersect transversely with the smooth points of $B$ if it intersects.

Step 1. Suppose that $\alpha\ne \ga_{max}$. Then the image of $\pi\circ \sigma$ intersects with more than two regions.
Take a path segment $L$ of $\pi(\sigma(I))\cap V$.
Let $P,Q$ be the end points of $L$ and assume that $P=\pi(\sigma(t_1))$ and 
$Q=\pi(\sigma(t_2))$ with $t_1<t_2$. By the assumption (2), $P,Q$ belongs to the  unique boundary component
$\partial_{+} V$ and there is a path $L''$ in the boundary $\partial_{+} V$ connecting $P,Q$ and 
$\gamma(\eta)\ge\alpha$ for any $\eta\in L''$. We want replace $L$ by some path $L'\subset V$ which is homotopic to $L''$ relative the end points.
See Figure \ref{deformation1}. Consider closed path at $Q$, $\omega:=L^{-1}\cdot L'$. 
The composition of paths is to be read from the left. Take a lift  $\tilde \omega$
which is a loop starting at $\sigma(t_2)$, passing through $\sigma(t_1)$ and comes back to $\si(t_2)$
 which is null homotopic in $\mathbb P^2\setminus C$.
We can simply take $\tilde \omega$ near the infinity.
Then replace $\sigma$ by $\sigma_{[0,t_2]}\cdot \tilde \omega\cdot \sigma_{[t_2,1]}$ which is homotopic to $\sigma$. Now
$\sigma$ is clearly homotopic to $\sigma'$ where
$\sigma':=\sigma_{[0,t_2]}\cdot\tilde\omega\cdot \sigma_{[t_2,1]}$.
 Note that  the image  $\pi(\sigma'(I))$   replace the segment $L$ by $L'$.
 Now we can deform $L'$ to $L''$ and further to the other side of the region of $L''$. Doing this operation for any path segment cut by $V$, we
  get a loop $\sigma''$ whose image by $\pi$ is in $\mathcal D_\beta$ where $\beta:=\min\{G\setminus\{\alpha\}\}$. By the inductive argument, we can deform $\sigma$ keeping the homotopy class to a  loop $ \sigma_1$ in $\pi^{-1}(U_{max})$.

\begin{figure}[htb]
\setlength{\unitlength}{1bp}

\begin{picture}(600,200)(-150,-630)        
{\includegraphics [width=6cm,height=4cm]
{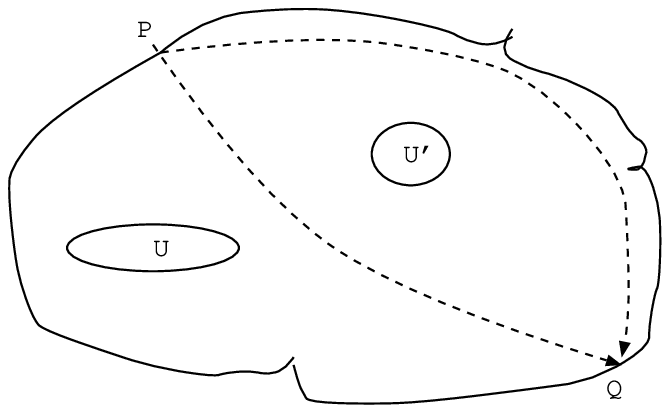}}
\put(-130,-470){$L''$}
\put(-130,-490){$L'$}
\put(-130,-550){$L$}
\put(-180,-530){$V$}
\end{picture}
\vspace{-2cm}
\caption{Segment $L$}\label{deformation1}
\end{figure}
Step 2. Now we assume that $\sigma_1$ is a loop $\pi\inv(U_{max})$. We deform $\sigma_1$  further
 to a loop $\hat \sigma$ which is 
a loop in the line $L_{\eta_0}$.

If $U_{max}$ is contractible, this is easy to deform using the fibration structure
of $\pi$ over $U_{max}$. This is the case for Example 1-1 and Example 2.
In Example 1-2, $U_{max}=T$ and $\pi_1(U_{max},\eta_0)$ is a free group of rank 2.

Assume that $\pi_1(U_{max})$ is non-trivial.
Put $\tau:=\pi\circ \sigma_1$ , a loop in $U_{max}$. Take a lift $\tilde \tau$  starting at $b$ which is  a contractible closed curve  in $\pi^{-1}U_{max}\setminus C$.
Consider the loop $\sigma_1\cdot {\tilde \tau}^{-1}$.
This is homotopic to $\sigma_1$. The image of this modified loop by $\pi$ is clearly homotopic to a constant loop at
 $\eta_0$.
Using the fibration structure over $U_{max}$, we can deform this loop to a loop $\hat{\sigma}$ in $L_{\eta_0}\setminus C$.
For the detail of lifting argument, see  for example Spanier \cite{Spanier}.
\end{proof}

The assertion is not true if $\eta_0$ does not belong to $U_{max}$. Also a loop $\tau\in (L_{\eta_0}\setminus C)$ can not be expressed
by a loop in $(L_\eta\setminus C)$ if $\gamma(\eta)<\gamma_{max}$ without using the monodromy relations. An example is given by $\xi_{2i-1},\xi_{2i}$ in Figure 2 can not deformed on the line $L_{\sigma(1/2+\eps)}$. We close this paper by  a question.
\newline\noindent
{Question.}
Does the conditions (1) and (2) hold for any mixed function?
\subsubsection{Appendix 1}\label{Appendix1}
Let $f$ be a mixed strongly  homogeneous polynomial. To compute the defining polynomial of the branching locus $R$ in Example 1-1, Example 1-2 and Example 2,
we proceed as follows.
Let $z=x+yi$ and $w=u+vi$ and write $f$ as $g+ih$ where $g,h$ are polynomials of $x,y,u,v$ with real coefficients.
Let $J=\frac{\partial g}{\partial x}\frac{\partial h}{\partial y}-\frac{\partial g}{\partial y}\frac{\partial h}{\partial x}$
and let $A=[g,h,J]$, the ideal generated by $g,h,J$.
Then we use the maple command:
{ Groebner[Basis](A,plex(x,y,u,v))}.
 For further explanation for Groebner calculation, we refer \cite{Cox} for example.

\vspace{.3cm}\noindent
{\bf Acknowledgement.} For the numerical calculation of roots of $f(z,w)=0$ with fixed various complex numbers $w$'s, we have used the following program on maple which is kindly 
written by Pho Duc Tai, Hanoi University of Science. I am grateful to him for his help.

\vspace{.3cm}
{\bf  Pho's program} to compute roots of mixed polynomial on Maple:\newline
fsol3 := proc (f, z)\nl
local aa, a, b, ff, f1, f2, h, i, j, k, s, temp;
print(Factorization\_of\_Input = factor(f)); ff := factors(f)[2]; temp := \{\};\nl
for k to nops(ff) do\nl
if 1 $<$ ff[k][2] then RETURN(printf("Input is not squarefree. Please solve each factor.")) end if;\nl
assume(a, real); assume(b, real); h := expand(subs(z = a+I*b, ff[k][1]));\nl
f1 := Re(h); f2 := Im(h); aa := RootFinding[Isolate]([f1, f2], [a, b]);\nl
temp := \`{}union\`{}(temp, {seq([[op(aa[i][1])][2], [op(aa[i][2])][2]], i = 1 .. nops(aa))}) end do; RETURN([op(temp)]) end proc

\subsubsection{Appendix2: Equation of R for Example 1.2}\label{eq-R}
The equation of the branching locus is the following.
\begin{multline*}
R(u,v)=-179685+129384576\,{u}^{4}{v}^{4}+2160\,{u}^{19}{v}^{2}+27\,{u}^{24}-
864\,{v}^{22}\\
+13590816\,{u}^{7}-102858240\,{v}^{6}{u}^{5}-47520\,{u}^{18}{v}^{4}-7631712\,u+174564288\,{v}^{6}{u}^{4}\\
-288581376\,{v}^{6}{u}^{2}+193050720\,{v}^{8}{u}^{2}+580608\,{v}^{14}{u}^{3}+2032128\,{u}^{13}{v}^{4}\\
+5080320\,{u}^{9}{v}^{8}+4064256\,{u}^{7}{v}^{10}+72576\,u{v}^{16}-142560\,{u}^{6}{v}^{16}\\
-9504\,{u}^{2}{v}^{20}-142560\,{u}^{16}{v}^{6}-399168\,{u}^{12}{v}^{10}-285120\,{u}^{14}{v}^{8}\\
-285120\,{u}^{8}{v}^{14}-47520\,{u}^{4}{v}^{18}-399168\,{u}^{10}{v}^{12}+441460992\,{v}^{2}{u}^{6}\\
+27\,{v}^{24}+542688\,{v}^{16}+55344648\,{v}^{6}{u}^{6}+216\,{u}^{21}+12096\,{v}^{20}-6048\,{u}^{19}\\
+20877120\,{u}^{2}-466968\,{u}^{15}+4064256\,{u}^{11}{v}^{6}+2032128\,{u}^{5}{v}^{12}+580608\,{u}^{15}{v}^{2}\\
+1550016\,{v}^{2}-21912872\,{u}^{3}-103122216\,{v}^{2}u+15611136\,u{v}^{10}-2415360\,{v}^{8}{u}^{3}\\
-212093856\,{u}^{3}{v}^{4}+165770304\,{u}^{7}{v}^{2}-93452\,{u}^{18}+33480480\,{v}^{8}{u}^{5}+
27856256\,{v}^{6}{u}^{3}\\
+138815904\,{v}^{6}u-137971200\,{u}^{7}{v}^{4}-46557076\,{v}^{6}+134691072\,{u}^{2}{v}^{2}+20322912\,{u}^{4}\\
+35835552\,{v}^{4}-15874316\,{u}^{6}+588672\,{u}^{5}-425273772\,{v}^{2}{u}^{4}+404256\,{u}^{16}\\
+352376064\,{v}^{2}{u}^{3}+44453280\,{v}^{4}{u}^{9}+57198720\,{v}^{6}{u}^{7}+6479040\,{v}^{10}{u}^{3}\\
-596640\,{v}^{12}u+6378146\,{v}^{12}+271049040\,{u}^{5}{v}^{4}-55084800\,{u}^{9}{v}^{2}-242721696\,{u}^{8}{v}^{2}\\
-283722816\,{u}^{6}{v}^{4}+85432284\,{u}^{10}{v}^{2}+15257280\,{u}^{11}{v}^{2}-9504\,{u}^{20}{v}^{2}\\
-73189752\,{v}^{8}u+324\,{u}^{22}{v}^{2}+1597920\,{u}^{13}+2630208\,{v}^{14}{u}^{2}-3738096\,{u}^{14}{v}^{4}\\
-12425640\,{u}^{10}{v}^{8}-11851032\,{u}^{8}{v}^{10}-8601296\,{u}^{12}{v}^{6}-7583152\,{u}^{6}{v}^{12}\\
-912372\,{u}^{16}{v}^{2}-752448\,{u}^{14}+190612542\,{u}^{8}{v}^{4}-52716324\,{u}^{2}{v}^{10}\\
-107902338\,{v}^{8}{u}^{4}+1782\,{u}^{20}{v}^{4}-50865792\,{v}^{4}u-4232728\,{u}^{9}+9720\,{u}^{17}{v}^{4}\\
+2160\,{u}^{3}{v}^{18}+9720\,{u}^{5}{v}^{16}+25920\,{u}^{15}{v}^{6}+54432\,{u}^{11}{v}^{10}+45360\,{u}^{9}{v}^{12}\\
+45360\,{u}^{13}{v}^{8}+25920\,{u}^{7}{v}^{14}+216\,u{v}^{20}+1451520\,{u}^{6}{v}^{14}+120960\,{u}^{2}{v}^{18}\\
+2155584\,{v}^{12}{u}^{2}-821676\,{u}^{2}{v}^{16}-18081480\,{u}^{7}{v}^{8}-11114040\,{u}^{11}{v}^{4}\\
-18630120\,{u}^{9}{v}^{6}-10126488\,{u}^{5}{v}^{10}-3003624\,{u}^{3}{v}^{12}-3552264\,{u}^{13}{v}^{2}\\
-3198096\,{u}^{4}{v}^{14}+72576\,{u}^{17}-337318944\,{u}^{5}{v}^{2}-99220\,{v}^{18}-2136768\,{v}^{14}\\
+544320\,{u}^{16}{v}^{4}+2540160\,{u}^{12}{v}^{8}+3048192\,{u}^{10}{v}^{10}+1451520\,{u}^{14}{v}^{6}\\
+2540160\,{u}^{8}{v}^{12}+120960\,{u}^{18}{v}^{2}+544320\,{u}^{4}{v}^{16}+12096\,{u}^{20}+80914380\,{v}^{4}{u}^{2}\\
+7271904\,{u}^{10}-864\,{u}^{22}+10619712\,{u}^{4}{v}^{10}+18579648\,{u}^{12}{v}^{4}-22564800\,{u}^{6}{v}^{8}\\+
33143040\,{u}^{8}{v}^{8}+19619136\,{u}^{6}{v}^{10}-78557760\,{u}^{8}{v}^{6}-71293248\,{u}^{10}{v}^{4}\\
+33409728\,{u}^{10}{v}^{6}+7934784\,{u}^{4}{v}^{12}+4945344\,{u}^{14}{v}^{2}+5940\,{u}^{18}{v}^{6}\\
+5940\,{u}^{6}{v}^{18}+324\,{u}^{2}{v}^{22}+13365\,{u}^{16}{v}^{8}+24948\,{u}^{12}{v}^{12}+21384\,{u}^{14}{v}^{10}\\
+13365\,{u}^{8}{v}^{16}+1782\,{u}^{4}{v}^{20}+21384\,{u}^{10}{v}^{14}-217728\,{u}^{15}{v}^{4}-762048\,{u}^{11}{v}^{8}\\
-762048\,{u}^{9}{v}^{10}-508032\,{u}^{13}{v}^{6}-508032\,{u}^{7}{v}^{12}-54432\,{u}^{17}{v}^{2}-54432\,{u}^{3}{v}^{16}\\
-217728\,{u}^{5}{v}^{14}-6048\,u{v}^{18}-7068480\,{u}^{8}-357240\,u{v}^{14}+30563520\,{v}^{8}\\-15242784\,{v}^{10}-1027742\,{u}^{12}-1945344\,{u}^{11}-22380096\,{u}^{12}{v}^{2}
\end{multline*}

\def\cprime{$'$} \def\cprime{$'$} \def\cprime{$'$} \def\cprime{$'$}
  \def\cprime{$'$} \def\cprime{$'$} \def\cprime{$'$} \def\cprime{$'$}


\begin{thebibliography}{1}
\bibitem{Bleher}
 P. Bleher, Y. Homma, L. Ji,
              P. Roeder.
    {\em Counting zeros of harmonic rational functions and its
              application to gravitational lensing},
   {Int. Math. Res. Not. IMRN},
    {8},
   {2245--2264}.
 \bibitem{Cox}
     {D.A.~Cox, J.~Little and D.~O'Shea}.
    {\em Ideals, varieties, and algorithms}.
     {UTM},
     {\em An introduction to computational algebraic geometry and
              commutative algebra}.
 {Springer, Cham},
     {2015}.
     \bibitem{Hamm-Le}
H. A. Hamm and  D.T. L\^e.
\newblock{\em  Un th\'eor\`eme de {Z}ariski du type de {L}efschetz}.
\newblock {Ann. Sci. \'Ecole Norm. Sup. (4)}, 6:317-355, 1973.
\bibitem{Molina}
J.~L. Cisneros-Molina.
\newblock {\em Join theorem for polar weighted homogeneous singularities.}
\newblock In {Singularities {II}}, volume 475 of {\em Contemp. Math.},
  pages 43--59. Amer. Math. Soc., Providence, RI, 2008.
\bibitem{Inaba}
{K. Inaba, M. Kawashima and M. Oka.}
{\em Topology of mixed hypersurfaces of cyclic type.}
to appear in {J. Math. Soc. Japan.}


\bibitem{Milnor}
J.~Milnor.
\newblock {\em Singular points of complex hypersurfaces}.
\newblock Annals of Mathematics Studies, No. 61. Princeton University Press,
  Princeton, N.J., 1968.
\bibitem{OkaSurvey}
 M. Oka.
 {\em A survey on {A}lexander polynomials of plane curves},
 in  {Singularit\'es {F}ranco-{J}aponaises},
   {S\'emin. Congr.},
    {10},
     {209--232},
 {Soc. Math. France, Paris},
   {2005}.
  
\bibitem{OkaPolar}
M.~Oka.
\newblock{\em Topology of polar weighted homogeneous hypersurfaces.}
\newblock { Kodai Math. J.}, 31(2):163--182, 2008.
\bibitem{OkaJarc}
  {M. Oka.}
    {\em On mixed plane curves of polar degree 1}.
  { The {J}apanese-{A}ustralian {W}orkshop on {R}eal and {C}omplex
              {S}ingularities---{JARCS} {III}},
    {Proc. Centre Math. Appl. Austral. Nat. Univ.},
    {43},
    {67--74},
     {2010}.
\bibitem{OkaBrieskorn}
 {M. Oka}.
{\em On mixed {B}rieskorn variety}.
  { Topology of algebraic varieties and singularities},
{Contemp. Math.},
    {538},
     {389--399},
  {Amer. Math. Soc., Providence, RI},
       {2011}.
\bibitem{MixIntersection}M.~Oka.
{\em Intersection theory on mixed curves.}
{Kodai J. Math.}
35 (2012), no. 2, 248-267.

 \bibitem {OkaLens}
M.~Oka.
\newblock {\em On the root of an extended Lens equation and an application.}
Math. arXiv 1505.03576v2,to appear in Singularities and Foliations. Geometry, Topology and Applications, 
	Salvador, Brazil,2015, Springer Proceedings in Mathematics \& Statistics


\bibitem{MC}
M.~Oka.
{\em On mixed projective curves}.
{Singularities in geometry and topology},
    {IRMA Lect. Math. Theor. Phys.},
    {20},
    {133--147},
  {Eur. Math. Soc., Z\"urich},
      {2012}.
  
\bibitem{OkaMix}
M.~Oka.
\newblock {\em Non-degenerate mixed functions.}
\newblock { Kodai Math. J.}, 33(1):1--62, 2010.


\bibitem{Rhie}
S.H.~Rhie.
\newblock {\em n-point Gravitational Lenses with 5(n-1) Images.}
\newblock { arXiv:astro-ph/0305166,} May 2003.
\bibitem{Spanier}
E.H. Spanier.
\newblock {\em Algebraic topology},
MacGraw Hill,1966.
\bibitem{Za1}
O. Zariski.
\newblock {\em On the problem of existence of algebraic functions of two variables
  possessing a given branch curve.}
\newblock {Amer. J. Math.}, 51:305-328, 1929.


\end{thebibliography}
\end{document}